\documentclass{article}
\usepackage[utf8]{inputenc}
\usepackage[T1]{fontenc}
\usepackage{amsmath,amssymb,amsfonts}
\usepackage{theorem}
\usepackage{color}
\usepackage{hyperref}

\theoremstyle{change}
\newtheorem{definition}{Definition:}[section]
\newtheorem{proposition}[definition]{Proposition:}
\newtheorem{theorem}[definition]{Theorem:}

\newtheorem{corollary}[definition]{Corollary:}
{\theorembodyfont{\rmfamily}
\newtheorem{remark}[definition]{Remark:}
}
{\theorembodyfont{\rmfamily}

}
\newenvironment{proof}
  {{\bf Proof:}}
  {\qquad \hspace*{\fill} $\Box$}
\newcommand{\fa}{\mathfrak{a}}
\newcommand{\fg}{\mathfrak{g}}
\newcommand{\fk}{\mathfrak{k}}

\newcommand{\fn}{\mathfrak{n}}

\newcommand{\ad}{\operatorname{ad}}

\newcommand{\inner}{\operatorname{int}}

\newcommand{\rme}{\mathrm{e}}

\newcommand{\SC}{\mathcal{S}}
\newcommand{\UC}{\mathcal{U}}

\newcommand{\XC}{\mathcal{X}}
\newcommand{\AC}{\mathcal{A}}
\newcommand{\DC}{\mathcal{D}}

\newcommand{\N}{\mathbb{N}}

\newcommand{\R}{\mathbb{R}}

\begin{document}

\title{A semigroup associated to a linear control system on a Lie group}
\author{V\'{\i}ctor Ayala\thanks{%
Supported by Proyecto Fondecyt n$%
{{}^\circ}%
$ 1150292. Conicyt, Chile.} \\
Instituto de Alta Investigaci\'{o}n\\
Universidad de Tarapac\'{a}\\
Casilla 7D, Arica, Chile \and Adriano Da Silva\thanks{Supported by Fapesp grant n${{}^\circ}$ 2016/11135-2.}\\
Instituto de Matem\'{a}tica\\
Universidade Estadual de Campinas\\
Cx. Postal 6065, 13.081-970 Campinas-SP, Brasil}
\date{\today}
\maketitle

\begin{abstract}
Let us consider a linear control system $\Sigma $ on a connected Lie group $%
G $. It is known that the accessibility set $\mathcal{A}$ from the identity $%
e$ is in general not a semigroup. In this article we associate a new
algebraic object $\mathcal{S}_{\Sigma }$ to $\Sigma $ which turns out to be
a semigroup, allowing the use of the semigroup machinery to approach $\Sigma$ In particular, we obtain some controllability results.
\end{abstract}

\textbf{Keywords: } linear control systems, semigroups, controllability, Lie groups.

\textbf{AMS 2010 subject classification}: 93C05, 93B05, 22E30. 

\section{Introduction\qquad}
The concept of linear control systems on Lie groups was introduced in \cite{VAJT} by Ayala and Tirao as the family of differential equations
\begin{equation}
\Sigma :\dot{g}(t)=\mathcal{X}(g(t))+\sum_{j=1}^{m}u_{j}(t)\text{ }%
X^{j}(g(t)),  \label{Linear system}
\end{equation}
where $X^{j}$ are right-invariant vector fields, $u\in \mathcal{U}\subset L^{\infty }(\mathbb{R},\Omega \subset \mathbb{R}^{m})$ is the class of
admissible controls, with $\Omega\subset\mathbb{R}^m$ a compact, convex subset satisfying $0\in \mathrm{int}\Omega$ and the drift $\XC$ is a linear vector field, that is, its associated flow $(\varphi _{t})_{t\in \mathbb{R}}$ is a 1-parameter subgroup of the group of the automorphisms $\mathrm{Aut}(G)$ of $G$. 

The study of linear control systems is important for at least two reasons: First, it is very well known that the classical linear system on the Euclidean space $\mathbb{R}^{d}$ given by 
$$\dot{x}(t)=Ax(t)+Bu(t), \;\;A\in\R^{d\times d}, B\in\R^{d\times m}\mbox{ and }u\in\UC$$
is one of the most relevant control systems. Such system can be written as 
\begin{equation*}
\dot{x}(t)=Ax(t)+\sum_{j\text{ }=1}^{m}u_{j}(t)\text{ }b^{j}, \mbox{ where } b^{j}\in \mathbb{R}^{d} \mbox{ are the columns of } B.
\end{equation*}%
Moreover, the flow associated with $A$ satisfies $e^{tA}\in GL^{+}(d, \mathbb{R)}=Aut(\mathbb{R}^{d})$ and, since $\mathbb{R}^{d}$ is a commutative Lie group, any constant vector $b^{j}$ is a right-invariant vector field, showing that the notion of linear systems on Lie groups is a generalization of linear systems on Euclidean spaces.

Secondly, in \cite{JPh1} Jouan shows that any control-affine system on a connected manifold $M$ whose dynamic generates a finite Lie algebra is diffeomorphic to a linear control system on a Lie group or on a homogeneous space, showing that the understanding of the behavior of the system $\Sigma$ is in fact very important in applications. 

One of the ways to analyze the dynamical behavior of control systems on Lie groups or homogeneous spaces is via semigroup theory by relating the solutions of the system with the action of a semigroup. For instance, for right-invariant systems the reachable set from the identity of the group is a semigroup and so all the machinery of the semigroup theory is available (see \cite{JuSu}). Different from right-invariant systems, the reachable set of a linear control system on a nonabelian Lie group cannot be a semigroup unless it is the whole group (see \cite{VASM}, Section 4). In this paper we show that, despite this fact, the reachable set strictly contains a semigroup that is intrinsically connected with the controllability properties of the system.

The paper is structured as follows: In Section 2 we introduce some $G$-subgroups associated to a given linear vector field. We define the reachable set of a linear system and show the relation between this set with the subgroups induced by its drift. In Section 3 we define the semigroup associated to a given linear system. In order to assure controllability of the system we show that it is enough to analyze the semigroup. Morever, if we assume that the reachable set is open our work reduces the problem from an arbitrary Lie group $G$ to a nilpotent $G$-subgroup. At the end of the section we show that for semisimple Lie groups the mentioned semigroup has nonempty interior if and only if it is the whole group if and only if the associated linear control system is controllable.

\section{Preliminaries}

In this section we give the background needed about linear control systems and group decomposition induced by linear vector field.

Let us assume from here on that $G$ is a connected Lie group with Lie algebra $\mathfrak{g}$ and dimension $d$, where $\fg$ is identified with the set of the right-invariant vector fields. For any linear vector field $\XC$ on $G$ let $\DC$ be the $\fg$-derivation determined by $\mathcal{X}$ (see for instance \cite{JPh}). The relation between $\DC$ and the flow $(\varphi_{t})_{t\in\R}$ of $\XC$ is given by  
$$(d\varphi_t)_e=\rme^{t\DC}, \;\;t\in\R.$$

Consider the complexification $\fg_{\mathbb{C}}$ of $\fg$ and the derivation $\DC_{\mathbb{C}}$ on $\fg_{\mathbb{C}}$ induced by $\DC$. For an eigenvalue $\alpha$ of $\DC$ we defined the $\alpha$-generalized eigenspace of $\DC$ in $\fg$ as 
$$\fg_{\alpha}:=\left(\fg_{\mathbb{C}}\right)_{\alpha}\cap\fg$$
where \begin{equation*}
(\fg_{\mathbb{C}})_{\alpha}=\{X\in \mathfrak{g}_{\mathbb{C}}:(\mathcal{D}_{\mathbb{C}}-\alpha)^{n}X=0\mbox{
for some }n\geq 1\}
\end{equation*}
is the $\alpha$-generalized eigenspace of $\DC_{\mathbb{C}}$ in $\fg_{\mathbb{C}}$. Since the eigenvalues of $\DC_{\mathbb{C}}$ coincides with the ones of $\DC$ and $\fg$ is $\DC_{\mathbb{C}}$-invariant, we have that any $\fg_{\alpha}$ is a $\DC$-invariant subspace of $\fg$ and $$\fg=\bigoplus_{\alpha}\fg_{\alpha}, \;\;\mbox{ for every }\alpha \mbox{ eigenvalue of }\DC.$$
Moreover, if $\beta$ is also an eigenvalue of $\DC$ Proposition 3.1 of \cite{SM1} implies
\begin{equation*}
[ (\fg_{\mathbb{C}})_{\alpha }, (\fg_{\mathbb{C}})_{\beta }]\subset (\fg_{\mathbb{C}})_{\alpha +\beta } \;\;\;\;\mbox{ and so } \;\;\;\;[ \fg_{\alpha }, \fg_{\beta }]\subset \fg_{\alpha +\beta }
\end{equation*}
where $(\fg_{\mathbb{C}})_{\alpha+\beta}=\fg_{\alpha+\beta}=\{0\}$ when $\alpha +\beta$ in not an eigenvalue of $\DC$.

 Therefore, $\mathfrak{g}$ decomposes as 
\begin{equation*}
\mathfrak{g}=\mathfrak{g}^{+}\oplus \mathfrak{g}^{-}\oplus \mathfrak{g}^{0}
\end{equation*}%
where 
\begin{equation*}
\mathfrak{g}^{+}=\bigoplus_{\alpha \;:\,\mathrm{Re}(\alpha )>0}\mathfrak{g}%
_{\alpha },\hspace{0.8cm}\mathfrak{g}^0=\bigoplus_{\alpha\;:\,\mathrm{Re%
}(\alpha)=0}\mathfrak{g}_{\alpha}\hspace{0.8cm}\mbox{ and }\hspace{0.8cm}%
\mathfrak{g}^{-}=\bigoplus_{\alpha \;:\,\mathrm{Re}(\alpha)<0}\mathfrak{g}%
_{\alpha}.
\end{equation*}%
The subspaces $\mathfrak{g}^{+},$ $\mathfrak{g}^{0}$ and $\mathfrak{g}^{-}$ are Lie algebras and $\mathfrak{g}^{+},$ $\mathfrak{g}^{-}$ are nilpotent. Moreover, $\fg^+$ and $\fg^-$ are ideals of the Lie subalgebras $\fg^{+, 0}:=\fg^+\oplus\fg^0$ and $\fg^{-, 0}:=\fg^-\oplus\fg^0$, respectively. Let us denote by, $G^{+}, G^{0}$, $G^{-}$, $G^{+, 0}$ and $G^{-, 0}$ the connected Lie subgroups of $G$ with Lie algebras $\mathfrak{g}^{+},$ $%
\mathfrak{g}^0$, $\mathfrak{g}^-$, $\fg^{+, 0}$ and $\fg^{-, 0}$ respectively.

By Proposition 2.9 of \cite{DaSilva} the subgroups $G^{+}, G^{0}$, $G^{-}$, $G^{+, 0}$ and $G^{-, 0}$ are closed subgroups of $G$ that are invariant by the flow $\left(\varphi_t\right)_{t\in\R}$.

From here, we focus in a linear control system $\Sigma $ on a connected Lie group $G$ as follows 
\begin{equation*}
\Sigma :\dot{g}(t)=\mathcal{X}(g(t))+\sum_{j\text{ }%
=1}^{m}u_{j}(t)X^{j}(g(t)),
\end{equation*}%
where $\mathcal{X}$ is a linear vector field, $X^{j}$ are right-invariant vector fields and $u\in\UC$. For any $g\in G$, $u\in\UC$ and $\tau \in \mathbb{R}$ we denote by $\phi _{\tau ,u}(g)$ the solution of $\Sigma$ at time $\tau$ with initial condition $g$ and control $u$. If $\tau > 0$ 
\begin{equation}
\mathcal{A}_{\tau }(g):=\{\phi _{\tau ,u}(g):u\in \mathcal{U}\}\;\;\;\text{%
and}\;\;\;\mathcal{A}(g):=\bigcup_{\tau >0}\mathcal{A}_{\tau }(g),
\label{reachablesets}
\end{equation}%
are the \textit{reachable set from }$g$\textit{\ at time }$\tau $ and the \textit{reachable set of }$g$, respectively. We also denote $\mathcal{A}_{\tau }(e)=\mathcal{A}%
_{\tau }$ and $\mathcal{A}(e)=\mathcal{A}$. We will say that the system $\Sigma$ satisfies the {\it rank}-condition if 
$$\mathrm{span}_{\mathcal{L}}\{\DC^i(X^j), i\in\N_0, j=1, \ldots, m\}=\fg,$$
that is, if the smallest Lie subalgebra of $\fg$ that is $\DC$-invariant coincides with $\fg$.

The next proposition states the main properties of the reachable sets. Its
proof can be found in \cite{JPh} Proposition 2.

\begin{proposition}
It holds:
\end{proposition}

\begin{itemize}
\item[$1.$] if $0\leq \tau _{1}\leq \tau _{2}$ then $\mathcal{A}_{\tau
_{1}}\subset \mathcal{A}_{\tau _{2}}$

\item[$2.$] for all $g\in G$ we have $\mathcal{A}_{\tau }(g)=\mathcal{A}%
_{\tau }\varphi _{\tau }(g)$

\item[$3.$] for all $\tau _{2},\tau _{1}\geq 0$ we have $\mathcal{A}_{\tau
_{1}+\tau _{2}}=$ $\mathcal{A}_{\tau _{1}}\varphi _{\tau _{1}}(\mathcal{A}%
_{\tau _{2}})=\mathcal{A}_{\tau _{2}}\varphi _{\tau _{2}}(\mathcal{A}_{\tau
_{1}})$
\end{itemize}

\begin{definition}
The system will be said to be controllable if $\AC=G$.
\end{definition}

In \cite{DaAy} the authors study the controllability property of linear control systems and introduce the following notion

\begin{definition}
\label{finitess} Let $G$ be a connected Lie group. We say that the Lie group $G$ has finite semisimple center if each semisimple Lie subgroup of $G$ has
finite center.
\end{definition}

\begin{remark}
For instance, a connected Lie group $G$ has finite semisimple center if $G$ has one (and hence all) Levi subgroup $L$ with finite center. Actually, any
semisimple Lie subgroup of $G$ is conjugated to a subgroup of $L$. These fact comes from Malcev's Theorem (see \cite{ALEB} Theorem 4.3) and its corollaries.
\end{remark}

Assume that $G$ has finite semisimple center. The relation between the subgroups induced by $\XC$ and the linear control system $\Sigma$ is given by the next result (see Theorem 3.9 of \cite{DaAy}) 

\begin{theorem}
\label{reachable}
If $\AC$ is open then $G^{+,0}\subset\AC$.
\end{theorem}

\begin{remark}
By Theorem 3.5 of \cite{VAJT}, one way to assure the openness condition on $\AC$ is by the {\it ad-rank} condition. The system $\Sigma$ is said to satisfies the ad-rank condition if 
$$\mathrm{span}\{\DC^i(X^j), i\in\N_0, j=1, \ldots, m\}=\fg, \;\;\mbox{ where }\;\;\N_0:=\N\cup\{0\}.$$
When $\Sigma$ satisfies the ad-rank condition we have that $e\in\inner\AC_{\tau}$ for all $\tau>0$ which implies, in particular, that $\AC$ is an open subset.
\end{remark}

\section{Semigroup associated to a linear control system on a Lie group}

Since the positive orbit $\mathcal{A}$ of a linear control system $\Sigma $ is in general not a semigroup (see for instance \cite{VASM}, Section 4), in this section we associate to $\Sigma $ a
new algebraic object $\mathcal{S}_{\Sigma }$ which turns out to be a semigroup. In particular, $\mathcal{S}_{\Sigma }$ enable us to pass from the
control theo\-ry of linear systems to the theory of semigroups. Furthermore, controllability of $\Sigma $ is equivalent to $\mathcal{S}_{\Sigma}=G$.

From here we assume that the linear system $\Sigma$ on $G$ satisfies the rank-condition. As before, we denote by $\left(\varphi _{t}\right)_{t\in\R}$ the $1$-parameter group of automorphisms associated to the drift $\mathcal{X}$ of $\Sigma$. Next, we introduce the subset $%
\mathcal{S}_{\Sigma }\subset G$ defined by 
\begin{equation*}
\mathcal{S}_{\Sigma }:=\bigcap_{t\in \mathbb{R}}\varphi _{t}(\mathcal{A}).
\end{equation*}%
Since $\varphi _{t}(e)=e$ for all $t\in \mathbb{R}$ and $e\in \mathcal{A}$
it follows that $\mathcal{S}_{\Sigma }\neq \emptyset $.

\begin{proposition}
\label{semigroup} With the previous notations it holds

\begin{itemize}
\item[1.] $\mathcal{S}_{\Sigma }$ is the greatest $\varphi $-invariant
subset of $\mathcal{A}$

\item[2.] For any $\tau _{0}\geq 0$ 
\begin{equation*}
\mathcal{S}_{\Sigma }=\bigcap_{t\text{ }\geq \tau _{0}}\varphi _{t}(%
\mathcal{A})
\end{equation*}

\item[3.] $x\in \mathcal{S}_{\Sigma }$ if and only if $\varphi _{t}(x)\in 
\mathcal{A}$ for all $t\leq 0$

\item[4.] $\mathcal{S}_{\Sigma }$ is a semigroup
\end{itemize}
\end{proposition}

\begin{proof}

\begin{itemize}
\item[1.] We start by proving the $\varphi $-invariance of $\mathcal{S}_{\Sigma }.$ Let $\tau \in \mathbb{R}$, then 
\begin{equation*}
\varphi _{\tau }(\mathcal{S}_{\Sigma })=\varphi _{\tau }\left(\bigcap_{t\in 
\mathbb{R}}\varphi _{t}(\mathcal{A)}\right)=\bigcap_{t\in \mathbb{R}}\varphi
_{\tau }(\varphi _{t}(\mathcal{A}))=\bigcap_{t\in \mathbb{R}}\varphi _{\tau
+t}(\mathcal{A})=\mathcal{S}_{\Sigma }
\end{equation*}%
Now, let $C$ be a $\varphi $-invariant subset of $\mathcal{A}$. It holds
that 
\begin{equation*}
C=\varphi _{t}(C)\subset \varphi _{t}(\mathcal{A}),\text{ for all}\;t\in 
\mathbb{R}\Leftrightarrow C\subset \bigcap_{t\in \mathbb{R}}\varphi _{t}(%
\mathcal{A})=\SC_{\Sigma }
\end{equation*}%
showing that $\SC_{\Sigma }$ is the greatest $\varphi $-invariant subset of $%
\mathcal{A}$.

\item[2.] By the $\varphi $-invariance of $\mathcal{A}$ in positive time, we
get 
\begin{equation*}
\tau _{0}-t\geq 0\Rightarrow \varphi _{\tau _{0}-t}(\mathcal{A})\subset \mathcal{%
A}\Leftrightarrow \varphi _{\tau _{0}}(\mathcal{A})\subset \varphi _{t}(%
\mathcal{A})
\end{equation*}%
and consequently 
\begin{equation*}
\varphi _{\tau _{0}}(\mathcal{A})=\bigcap_{t\text{ }\leq \text{ }\tau
_{0}}\varphi _{t}(\mathcal{A}).
\end{equation*}%
Therefore, 
\begin{equation*}
\mathcal{S}_{\Sigma }=\bigcap_{t\text{ }>\tau _{0}}\varphi _{t}(\mathcal{A}%
)\cap \bigcap_{\text{ }t\text{ }\leq \tau _{0}}\varphi _{t}(\mathcal{A}%
)=\bigcap_{t\text{ }>\tau _{0}}\varphi _{t}(\mathcal{A})\cap \varphi _{\tau
_{0}}(\mathcal{A})=\bigcap_{t\text{ }\geq \tau _{0}}\varphi _{t}(\mathcal{A}%
)
\end{equation*}%
as desired.

\item[3.] We have 
\begin{equation*}
x\in \mathcal{S}_{\Sigma }\Leftrightarrow x\in \varphi _{t}(\mathcal{A})%
\text{ for all }t\geq 0\Leftrightarrow \varphi _{t}(x)\in \mathcal{A}\text{
for all }t\leq 0.
\end{equation*}

\item[4.] Let $x,y\in \mathcal{S}_{\Sigma }$ and consider $x_{t}:=\varphi_{-t}(x)$ and $y_{t}=\varphi _{-t}(y)$. By item 3. we just need to show that 
$x_{t}y_{t}=\varphi_{-t}(x)\varphi_{-t}(y)=\varphi _{-t}(xy)\in \mathcal{A}$ for any $t\geq 0$. Since by
hypothesis $x_{t}\in \mathcal{A}$, there exists $s_{t}>0$ such that $%
x_{t}\in \mathcal{A}_{s_{t}}$. Since $y\in \mathcal{S}_{\Sigma }$ we get $%
\varphi _{-s_{t}}(y_{t})=\varphi _{-s_{t}-t}(y)\in \mathcal{A}$ and there is 
$s_{t}^{\prime }>0$ such that $\varphi _{-s_{t}}(y_{t})\in \mathcal{A}%
_{s_{t}^{\prime }}$ which gives us 
\begin{equation*}
x_{t}y_{t}=x_{t}\varphi _{s_{t}}(\varphi _{-s_{t}}(y_t))\in \mathcal{A}%
_{s_{t}}\varphi _{s_{t}}(\mathcal{A}_{s_{t}^{\prime }})=\mathcal{A}%
_{s_{t}+s_{t}^{\prime }}\subset \mathcal{A}
\end{equation*}%
showing that $S_{\Sigma }$ is a semigroup.
\end{itemize}
\end{proof}

\begin{definition}
$\mathcal{S}_{\Sigma }$ is called the semigroup of the system $\Sigma $.
\end{definition}

The next result states that, in order to study the controllability property of $\Sigma $ it is enough to look at the semigroup $\SC_{\Sigma }$. In fact,

\begin{theorem}
\label{equivalence}
$\mathcal{A}=G$ if and only if $\SC_{\Sigma }=\mathcal{A}.$
\end{theorem}

\begin{proof}
If $\mathcal{A}=G$, it follows that $\varphi _{t}(\mathcal{A})=\varphi
_{t}(G)=G$ for all $t\in \mathbb{R}$ which implies that 
\begin{equation*}
G=\bigcap_{t\text{ }\geq \text{ }0}\varphi _{t}(\mathcal{A})=\SC_{\Sigma }.
\end{equation*}

Reciprocally, if $\SC_{\Sigma }=\mathcal{A}$ then $\mathcal{A}$ is a
semigroup. However, Proposition 7 of \cite{JPh} assures that $\mathcal{A}$
is a semigroup if and only if $\mathcal{A}=G$. Then, the result follows.
\end{proof}

\begin{remark}
Just observe that from our previous theorem, $\mathcal{S}_{\Sigma }$ is not a proper semigroup of the reachable set if and only if it is the whole group 
$G$. In parti\-cular, the semigroup of the system enable us to pass from the control theory of linear systems to the theory of semigroups. Furthermore,
controllability of $\Sigma $ is equivalent to the equality $\mathcal{S}_{\Sigma }=G$.
\end{remark}

The above theorem implies the following.

\begin{corollary}
$\mathcal{A}=G\Leftrightarrow \mathcal{A}$ is $\varphi $-invariant.
\end{corollary}

\begin{proof}
If $\mathcal{A}=G$, then $\mathcal{A}$ is certainly $\varphi $-invariant.
Conversely, if $\varphi _{t}(\mathcal{A})=\mathcal{A}$ for any $t\in \mathbb{%
R}$ we obtain $\mathcal{A}$ $\subset \SC_{\Sigma }$ which implies $\SC_{\Sigma }=%
\mathcal{A}$ and by Theorem \ref{equivalence} we get $\mathcal{A}=G$.
\end{proof}

Under the assumption that the reachable set is open we get the following.

\begin{proposition}
If $\AC$ is open then $G^{+,0}\subset \SC_{\Sigma }$.
Moreover, 
\begin{equation*}
\Sigma \text{ is controllable}\Leftrightarrow G^{-}\subset \SC_{\Sigma }.
\end{equation*}
\end{proposition}

\begin{proof}
Since $\AC$ is open, Theorem \ref{reachable} implies $G^{+,0}\subset\AC$ and since $G^{+, 0}$ is $\varphi$-invariant we have by Proposition \ref{semigroup} that $G^{+,0}\subset \mathcal{S}_{\Sigma }$.

If $\Sigma$ is controllable, Theorem \ref{equivalence} implies $\SC_{\Sigma}=G$ and therefore $G^-\subset \SC_{\Sigma}$. Reciprocally, if we assume that $G^{-}\subset \SC_{\Sigma }$, we obtain $G^{+,0}G^{-}\subset \SC_{\Sigma }$. Since $G^{+,0}G^{-} $ is a neighborhood of the identity we get $e\in \inner \SC_{\Sigma }.$ On the other hand, we are assuming that $G$ is connected and
since we already prove that $\SC_{\Sigma }$ is a semigroup, it follows that $%
\SC_{\Sigma }=G$.
\end{proof}

The previous proposition shows that, in order to analyze the controllability
of $\Sigma $, it is enough to study the semigroup induced by $G^{-}$ given
by 
\begin{equation*}
\SC_{\Sigma }^{-}=\SC_{\Sigma }\cap G^{-}=\bigcap_{t\text{ }\geq \text{ }%
0}\varphi _{t}(\mathcal{A}\cap G^{-}).
\end{equation*}%
In other words, in order to analyze controllability of a linear control system on a connected Lie group $G$, we only need to study the behavior of the $\varphi$-invariant semigroup $\SC_{\Sigma}^-$ on the nilpotent Lie group $G^{-}$.

\begin{corollary}
Let $\Sigma $ be a linear control system on a connected Lie group $G$ such that $G^-=\{e\}$ and assume that $\Sigma $ satisfy the $ad$-rank condition. Then, $\Sigma $ is controllable
\end{corollary}

\begin{proof}
Since $G^{-}$ is trivial and $\mathcal{A}$ is open, the results follows
directly from the previous theorem.

\end{proof}

For compact Lie groups we have the following.

\begin{proposition}
Let $G$ be a compact Lie group. It holds that
$$\AC \mbox{ is open }\Leftrightarrow \AC=G \Leftrightarrow \SC_{\Sigma}=G$$
\end{proposition}

\begin{proof}
By Proposition 2.10 of \cite{DaSilva} if $G$ is compact then $G=G^0$. Therefore, by Theorem \ref{reachable} we have that $\AC$ is open if and only if $\AC=G$. The other equivalence is Theorem \ref{equivalence}. 
\end{proof}

\subsection{The semisimple case}
In this section we show that for linear systems on semisimple Lie groups the controllability of the linear control system $\Sigma$ is equivalent to $\inner\SC_{\Sigma}\neq\emptyset$. For the theory of semisimple Lie groups the reader can consult \cite{DuKoVa}, \cite{He}, \cite{Wa}.

Let $G$ be a connected non-compact semisimple Lie group $G$ with finite center and Lie algebra $\mathfrak{g}$. Fix a Cartan involution $\zeta:\mathfrak{g}\rightarrow \mathfrak{g}$ with associated Cartan decomposition $\mathfrak{g}=\mathfrak{k}\oplus \mathfrak{s}$ and let $\mathfrak{a}\subset \mathfrak{s}$ be a maximal abelian subalgebra and $\fa^+\subset \fa$ a Weyl chamber. Let us denote by $\Pi$, $\Pi^+$ and $\Pi^+:=-\Pi^+$ the {\bf set of roots}, the {\bf set of positive roots} and {\bf set of negative roots}, respectively, associated with $\fa^+$.
The {\bf Iwasawa decomposition} of the Lie algebra $\mathfrak{g}$, associated with the above choices, reads as 
\begin{equation*}
\mathfrak{g}=\mathfrak{k}\oplus \mathfrak{a}\oplus \mathfrak{n}^{\pm }\;\;\text{
where }\;\;\mathfrak{n}^{\pm }:=\sum_{\alpha \text{ }\in \text{ }\Pi ^{\pm }}%
\mathfrak{g}_{\alpha }.
\end{equation*}
and $\fg_{\alpha}:=\{X\in\fg, \;\;[H, X]=\alpha(H)X, \;\;\mbox{ for all }\;H\in\fa\}.$ Let $K$, $A$ and $N^{\pm }$ be the connected Lie subgroups of $G$ with Lie algebras $\mathfrak{k}$, $\mathfrak{a}$ and $\mathfrak{n}^{\pm }$,
respectively. The Iwasawa decomposition of the Lie group $G$ is given by $G=KAN^{\pm }$. If we denote by $M$ the centralizer of $\mathfrak{a}$ in $K$ it follows that $M$ is compact and $\sigma=AN^+MN^-$ is an open dense submanifold of $G$ (see \cite{He}, Chapter IX, Corollary 1.8).

Let us now consider a linear system $\Sigma$ on $G$ and let us denote by $\DC$ the derivation associated with the drift $\XC$. Since $\mathfrak{g}$ is semisimple the derivation $\DC$ is inner, that is, there exists $X\in \mathfrak{g}$ such that $\mathcal{D}$ = $\ad(X)$. Following \cite{He} Chapter IX Paragraph 7, there exists an Iwasawa decomposition $\mathfrak{g}=\fk\oplus \mathfrak{a}\oplus \mathfrak{n}^+$ and commuting elements $E\in \mathfrak{k}$, $H\in \fa$ and $N\in \mathfrak{n}^+$ such that $X=E+H+N$. Since $K$ is compact and $N^+$ nilpotent, $\ad(E)$ has only pure imaginary eigenvalues and $\ad(N)$ is nilpotent, implying that the real part of the eigenvalues of $\DC$ coincides with the eigenvalues of $\ad(H)$.

According to the previous analysis, we are able to prove the following

\begin{theorem}
\label{semisimple}
Let $G$ be a noncompact semisimple connected Lie group with finite center. Under the assumption that $\AC$ is an open set, it turns out that
$$\Sigma \mbox{ is controllable if and only if } \inner\mathcal{S}_{\Sigma}\neq\emptyset.$$
\end{theorem}

\begin{proof}
If $\Sigma$ is controllable then by Theorem \ref{equivalence} $\SC_{\Sigma}=G$ and so $\inner\SC_{\Sigma}$ is certainly nonempty. Reciprocally, let us assume that $\inner\SC_{\Sigma}\neq\emptyset$. By the previous analysis $\ad(H)$ restricted to $\fa\oplus\fn^+$ has only nonnegative eigenvalues, therefore we have that $\fa\oplus\fn^+\subset\fg^{+, 0}$ and consequently $AN^+\subset G^{+, 0}$. Moreover, since $\sigma=AN^+MN^-$ is a dense subset of $G$ and $G^{+, 0}\subset\SC_{\Sigma}$ we have that 
$$\inner\SC_{\Sigma}\cap MN^-\neq\emptyset.$$
Define
$$L:=\{m\in M; \exists n\in N^-\;\mbox{ with } mn\in \inner\SC_{\Sigma}\}.$$
The subset $L$ is nonempty and since $M$ normalizes $N^-$ it is a semigroup of $M$ with nonempty interior. Being that $M$ is compact we must have $M_0\subset L$, where $M_0$ stands for the connected component of the identity of $M$. Therefore, $\inner\SC_{\Sigma}\cap N^-\neq \emptyset$. However, Lemma 4.1 of \cite{SM} asserts that the only semigroup of $G$ that contains a nilpotent element in its interior is $G$ itself. Thus $\SC_{\Sigma}=G$ concluding the proof.
\end{proof}

The next result gives us another controllability characterization.

\begin{proposition}
Under the hypothesis of Theorem \ref{semisimple} it holds that $\inner\mathcal{S}_{\Sigma}\neq\emptyset$ if and only if there exists a $%
\varphi$-invariant open set $U\subset\mathcal{A}$.
\end{proposition}

\begin{proof}
In fact, if $\inner\mathcal{S}_{\Sigma}\neq\emptyset$ then $U=\inner\mathcal{S}_{\Sigma}$ satisfies the above condition. Reciprocally, if there exists a $\varphi$-invariant open set $U$ contained in $\AC$ we must have by item 1. of
Proposition \ref{semigroup} that $U\subset\mathcal{S}_{\Sigma}$ and
consequently $\inner\mathcal{S}_{\Sigma}\neq\emptyset$ as desired.
\end{proof}

\begin{remark}
We should notice that the open set $U\subset \mathcal{A}$ contains the identity if and only if $\SC_{\Sigma}=G$.
\end{remark}

\end{document}